\documentclass[11pt,a4paper,twoside]{article}

\usepackage{iftex}
\usepackage[T1]{fontenc}
\ifPDFTeX
  \usepackage[utf8]{inputenc}
\fi
\usepackage{lmodern}

\usepackage{amsmath,amsthm,amssymb,mathtools,esint}

\usepackage{geometry}
\usepackage{xcolor}
\usepackage{microtype}
\usepackage{fancyhdr}

\usepackage[
  colorlinks=true,
  linkcolor=blue!45!black,
  citecolor=blue!45!black,
  urlcolor=blue!45!black
]{hyperref}

\allowdisplaybreaks[3]
\newtheoremstyle{paperplain}{0.65em}{0.65em}{\itshape}{}%
  {\bfseries}{.}{0.5em}{}
\theoremstyle{paperplain}
\newtheorem{theorem}{Theorem}[section]
\newtheorem{proposition}[theorem]{Proposition}
\newtheorem{lemma}[theorem]{Lemma}
\newtheorem{corollary}[theorem]{Corollary}
\newtheoremstyle{paperremark}{0.65em}{0.65em}{\normalfont}{}%
  {\bfseries}{.}{0.5em}{}
\theoremstyle{paperremark}

\newcommand{\runningauthors}{XiaoRui Lu, Jin Yan, Bao Yu and Yang Zhou}
\newcommand{\runningtitle}{}
\let\originaltitle\title
\renewcommand{\title}[2][]{%
  \originaltitle{#2}%
  \if\relax\detokenize{#1}\relax
    \gdef\runningtitle{#2}%
  \else
    \gdef\runningtitle{#1}%
  \fi
}

\numberwithin{equation}{section}
\DeclareMathOperator{\tr}{tr}
\DeclareMathOperator{\dist}{dist}
\DeclareMathOperator{\diag}{diag}
\DeclareMathOperator{\supp}{supp}
\newcommand{\R}{\mathbb R}
\title[Sharp planar Monge--Amp\`ere estimates]{\textbf{Sharp Hessian integrability for the planar Monge--Amp\`ere equation}}
\author{\textbf{XiaoRui Lu}\and\textbf{Jin Yan}\and\textbf{Bao Yu}\and\textbf{Yang Zhou}}
\date{}

\makeatletter
\renewcommand{\@maketitle}{%
  \newpage\null
  \vskip .5em
  \begin{center}%
    {\LARGE\@title\par}%
    \vskip 1.2em
    {\large\lineskip .5em
      \begin{tabular}[t]{c}\@author\end{tabular}\par}%
  \end{center}%
  \par\vskip .7em
}
\newcommand{\printfirstpagebottommatter}{%
  \begingroup
  \renewcommand{\thefootnote}{}%
  \renewcommand{\@makefntext}[1]{\noindent##1}%
  \footnotetext[0]{%
    \textit{Keywords.}
    Monge--Amp\`ere equation; Hessian integrability;
    weak Lorentz estimate; adjoint equation; convexity.
    \par\smallskip
    \textit{2020 Mathematics Subject Classification.}
    35J96, 35B65.%
  }%
  \endgroup
}
\makeatother

\newcommand{\printauthorinformation}{%
  \par\bigskip
  \begingroup
  \small
  \setlength{\parindent}{0pt}%
  \textbf{XiaoRui Lu}\par
  School of Mathematical Sciences, Zhejiang University, Hangzhou 310058, China.\par
  \textit{Email:}
  \href{mailto:12435035@zju.edu.cn}{12435035@zju.edu.cn}.
  \par\medskip
  \textbf{Jin Yan}\par
  Institute of Mathematics,
  Academy of Mathematics and Systems Science,
  Chinese Academy of Sciences,
  55 Zhongguancun East Road, Beijing 100190, China.\par
  \textit{Email:}
  \href{mailto:yanjin@amss.ac.cn}{yanjin@amss.ac.cn}.
  \par\medskip
  \textbf{Bao Yu}\par
  School of Mathematical Sciences,
  University of Science and Technology of China,
  Hefei, Anhui Province, P. R. China.\par
  \textit{Email:}
  \href{mailto:baoyu1@mail.ustc.edu.cn}{baoyu1@mail.ustc.edu.cn}.
  \par\medskip
  \textbf{Yang Zhou}\par
  School of Mathematical Sciences,
  University of Science and Technology of China,
  Hefei, Anhui Province, P. R. China.\par
  \textit{Email:}
  \href{mailto:zy19700816@mail.ustc.edu.cn}{zy19700816@mail.ustc.edu.cn}.
  \par\medskip
  \endgroup
}

\hypersetup{%
  pdftitle={Sharp Hessian integrability for the planar Monge-Ampere equation},
  pdfauthor={XiaoRui Lu, Jin Yan, Bao Yu and Yang Zhou},
  pdfsubject={Optimal weak Hessian estimates for planar Alexandrov solutions},
  pdfkeywords={Monge-Ampere equation, Hessian integrability, Lorentz space, adjoint equation}
}

\begin{document}
\maketitle
\printfirstpagebottommatter

\begin{abstract}
We determine the optimal local integrability of the Hessian for planar
convex Alexandrov solutions satisfying
$$
 \lambda\,dx\leq\mu_u\leq\lambda^{-1}\,dx,
 \qquad 0<\lambda<1.
$$
The distributional Hessian satisfies the sharp
endpoint estimate
$$
 D^2u\in L_{\mathrm{loc}}^{P,\infty},
 \qquad
 P=\frac{1+\lambda}{1-\lambda}.
$$
This settles the weak
estimate proposed by De Philippis and Figalli and identifies the larger optimal exponent $P$. We further construct strictly convex $C^1$ solutions
for which the distribution function of the Laplacian is comparable to
$T^{-P}$ for all sufficiently large $T$. Hence both the exponent
and the weak Lorentz endpoint are optimal:
$L^{P,\infty}$ cannot be replaced by $L^{P,b}$ for any finite
$b>0$.

As applications, we obtain sharp estimates in Minty coordinates and
optimal quantitative bi-Lipschitz restrictions for gradient maps,
including an optimal exceptional-mass bound in quadratic-cost
transport.
\end{abstract}

\section{Introduction}
In this paper, we study convex solutions of the planar
Monge--Amp\`ere equation
\begin{equation}\label{eq:equation}
 \det D^2u=f\quad\text{in }\Omega\subset\R^2,
 \qquad 0<\lambda\leq f\leq\lambda^{-1}.
\end{equation}
The equation is understood in the Alexandrov sense: the measure
\(\mu_u\), defined by
\(\mu_u(E)=|\partial u(E)|\) for Borel sets \(E\), satisfies
\(\mu_u=f\,dx\). Here \(\partial u(E)\) is the union of the
subdifferentials over \(E\). Since the equation controls the product
of the Hessian eigenvalues, a central question is how much
integrability of the Hessian follows from the two bounds on \(f\)
alone.

For strictly convex Alexandrov solutions in arbitrary dimension,
Caffarelli \cite{CaffarelliLocalization,CaffarelliC1a} established
interior \(C^{1,\alpha}\) regularity under positive upper and lower
bounds on the density. If the density is continuous, his perturbation
theory \cite{CaffarelliW2p} further yields
\(W^{2,p}_{\mathrm{loc}}\) regularity for every finite \(p\).
For merely measurable densities bounded above and below by positive
constants, De Philippis and Figalli \cite{DFW21} proved
\[
 |D^2u|\in L\log^k L_{\mathrm{loc}}
 \qquad\text{for every finite }k,
\]
and hence \(u\in W^{2,1}_{\mathrm{loc}}\). De Philippis, Figalli and
Savin \cite{DFSW21} and, independently, Schmidt \cite{Schmidt}
subsequently obtained \(W^{2,1+\varepsilon}_{\mathrm{loc}}\)
regularity for some \(\varepsilon>0\) depending on the dimension and
the density bounds.

The role of strict convexity depends on the dimension. In dimensions
\(n\geq3\), a positive lower bound on the Monge--Amp\`ere measure does
not by itself imply strict convexity. For possibly non-strictly convex
solutions, Mooney \cite{MooneyPartial2015} proved partial regularity
and local \(W^{2,1}\) regularity, and later obtained quantitative
logarithmic estimates \cite{MooneyW212015}.

In dimension two, by contrast, the positive lower density bound
already implies strict convexity; see, for instance,
\cite{CaffarelliLocalization,Heinz}. Before the general
\(W^{2,1}\) theorem, Ambrosio, De Philippis and Kirchheim
\cite{ADK2011} related the planar Sobolev regularity problem to a
rigidity question for differential inclusions. De Philippis, Guerra
and Tione \cite{DGT2026} later used their unique-continuation theory
for differential inclusions to obtain another proof of planar
\(W^{2,1+\varepsilon}_{\mathrm{loc}}\) regularity through Minty's
correspondence and quasiregular mappings.

The available exponent must nevertheless depend on the density
contrast. For every \(p>1\), Wang \cite{WangCounterexamples}
constructed a two-dimensional example with
\[
 0<\lambda(p)\leq f\leq\lambda(p)^{-1}
\]
whose solution does not belong to \(W^{2,p}_{\mathrm{loc}}\).
The positive lower density bound is also essential: Mooney
\cite{MooneyDegenerate2016} constructed a planar convex Alexandrov
solution with bounded nonnegative density and zero boundary values
for which \(W^{2,1}\) regularity fails. Thus the existing planar
theory establishes higher integrability, but does not determine the
optimal exponent for a fixed positive density interval.

In dimension two, mixed gradient coordinates connect the
Monge--Amp\`ere equation with planar quasiconformal mappings.
For a smooth solution, De Philippis and Figalli
\cite[Section~5.2, problem~(2)]{DFSurvey} observed that the associated
mixed-coordinate map is \(1/\lambda^2\)-quasiconformal.
Combining this observation with the sharp planar quasiconformal
estimates of Astala \cite{Astala1994}, as presented in Astala,
Iwaniec and Martin \cite{AstalaIwaniecMartin}, De Philippis and
Figalli conjectured that the sharp endpoint should be
\begin{equation}\label{eq:figalli}
 D^2u\in L_{\mathrm{loc}}^{p_*,\infty},
 \qquad
 p_*=\frac{1+\lambda^2}{1-\lambda^2}.
\end{equation}
Here weak \(L^p\) means that
\[
 \sup_{T>0}T^p
 |\{|D^2u|>T\}\cap\Omega'|<\infty
\]
for every \(\Omega'\Subset\Omega\).

We prove a stronger estimate and determine the optimal exponent:
throughout the paper,
\begin{equation}\label{eq:exponent}
 0<\lambda<1,\qquad
 P=\frac{1+\lambda}{1-\lambda}.
\end{equation}

\begin{theorem}\label{thm:main}
Let \(\Omega\subset\R^2\) be convex and open, and let \(u\) be convex
on \(\Omega\). Suppose
\begin{equation}\label{eq:measure}
 \lambda\,dx\leq\mu_u\leq\lambda^{-1}\,dx.
\end{equation}
Then the distributional Hessian of \(u\) has no singular part and
\(D^2u\in L^{P,\infty}_{\mathrm{loc}}(\Omega)\).
More precisely, suppose also that
\[
 B_1\subset\Omega\subset B_2,\qquad
 u\in C(\overline\Omega),\qquad
 u=0\quad\text{on }\partial\Omega.
\]
For \(h=-\min_\Omega u\) and
\(S_\theta=\{u<-(1-\theta)h\}\), one has
\begin{equation}\label{eq:weak}
 \sup_{T>0}T^P|\{\Delta u>T\}\cap S_\theta|
 \leq C(\lambda,\theta),\qquad 0<\theta<1.
\end{equation}
In particular,
\(u\in W^{2,p}_{\mathrm{loc}}(\Omega)\) for every \(1\leq p<P\).
\end{theorem}

The relation to \eqref{eq:figalli} is explicit:
\[
 P-p_*=\frac{2\lambda}{1-\lambda^2}>0.
\]
Theorem~\ref{thm:main} therefore proves the weak estimate asked for
in \cite[Section~5.2, problem~(2)]{DFSurvey}, and also proves strong
\(L^{p_*}\) integrability. The next theorem shows that \(P\) is the
optimal exponent.

\begin{theorem}\label{thm:sharp}
For every \(0<\lambda<1\), there is a strictly convex
\(u\in C^1(\R^2)\) satisfying \eqref{eq:measure}, equal to
\(|x|^2/2\) outside \(Q=(0,1)^2\), such that
\begin{equation}\label{eq:two-sided}
 c_\lambda T^{-P}\leq|\{\Delta u>T\}\cap Q|
 \leq C_\lambda T^{-P}\qquad(T\geq T_\lambda).
\end{equation}
The function is smooth outside a compact set of area zero, and
\(D^2u\notin L^{P,b}(Q)\) for every finite \(b>0\).
Moreover,
\(\|u-|x|^2/2\|_{C^1(\R^2)}\) can be arbitrarily small.
\end{theorem}

In particular, strong \(L^P\) regularity fails in general, even for
\(C^1\) solutions that agree with a quadratic polynomial outside a
square. For arbitrary density bounds \(0<\alpha<\beta\), replacing
\(u\) by \((\alpha\beta)^{-1/4}u\) shows that the optimal exponent is
\[
 \frac{\sqrt\beta+\sqrt\alpha}
      {\sqrt\beta-\sqrt\alpha}.
\]

The weak endpoint also has the following logarithmic consequence.

\begin{corollary}\label{cor:log}
Under the normalized hypotheses of Theorem~\ref{thm:main}, for
every \(\gamma>1\),
\begin{equation}\label{eq:logarithmic}
 \int_{S_\theta}
 \frac{(\Delta u)^P}
 {[\log(e+\Delta u)]^\gamma}\,dx
 \leq C(\lambda,\theta,\gamma).
\end{equation}
The restriction \(\gamma>1\) is optimal: for the example in
Theorem~\ref{thm:sharp}, the corresponding integral over \(Q\) is
finite if and only if \(\gamma>1\).
\end{corollary}

The endpoint estimates have two further sharp geometric
consequences. The first concerns Minty's parametrization of monotone
graphs. Minty's correspondence between monotone operators and
Lipschitz graphs \cite{Minty1962} provides the geometric origin of
the following coordinates; see also the measure-theoretic theory of
monotone maps developed by Alberti and Ambrosio
\cite{AlbertiAmbrosio1999}. This framework was used for planar
Monge--Amp\`ere regularity in \cite{ADK2011,DGT2026}.

Identify \(\R^2\) with \(\mathbb C\), and define
\begin{equation}\label{eq:minty-definition}
 S(x)=x+\nabla u(x),\qquad
 V=S(\Omega),\qquad
 g(S(x))=\overline{\nabla u(x)-x}.
\end{equation}
Here the bar denotes complex conjugation and
\[
 g_z=\frac12(\partial_{z_1}-i\partial_{z_2})g
\]
is the Wirtinger derivative.

\begin{corollary}\label{cor:minty}
Let \(\Omega\subset\R^2\) be convex and open, and let \(u\) be convex
on \(\Omega\) satisfying \eqref{eq:measure}. Let \(S,V,g\)
be defined by \eqref{eq:minty-definition}. Then
\(S:\Omega\to V\) is a homeomorphism and \(g\) is \(1\)-Lipschitz.
The defect
\[
 \delta=1-|g_z|
\]
is positive almost everywhere in \(V\). For every open set
\(V'\Subset V\),
\begin{equation}\label{eq:minty-upper}
 |\{z\in V':\delta(z)<t\}|
 \leq C\,t^{P-1}
 \quad(0<t<1),\qquad
 \delta^{-1}\in L^{P-1,\infty}(V').
\end{equation}
For each \(0<\lambda<1\), there exist a solution of
\eqref{eq:measure} and an open set \(V'\Subset V\) such that
\begin{equation}\label{eq:minty-sharp}
 c\,t^{P-1}\leq
 |\{z\in V':\delta(z)<t\}|
 \leq C\,t^{P-1}
 \qquad(0<t<t_0).
\end{equation}
For this example,
\(\delta^{-1}\notin L^{P-1,b}(V')\) for every \(0<b<\infty\).
Thus both the power and the weak Lorentz endpoint are optimal.
\end{corollary}

Theorem~\ref{thm:main} yields the exponent \(P-1\) in
\eqref{eq:minty-upper}; the loss of one power relative to \(P\)
comes from the Jacobian of the Minty map.
Theorem~\ref{thm:sharp} shows that both the exponent and the weak
Lorentz endpoint are optimal.

The second consequence is a quantitative Lusin-type restriction of
the gradient map. The pointwise inequalities of Bojarski and
Haj\l asz \cite{BojarskiHajlasz1993} bound Sobolev increments by
maximal functions of weak derivatives. Applying this mechanism to
both \(\nabla u\) and its Legendre-dual inverse leads to the following
result.

\begin{corollary}\label{cor:bilip}
Let \(\Omega\subset\R^2\) be convex and open, and let \(u\) be convex
on \(\Omega\) satisfying \eqref{eq:measure}. Set \(T=\nabla u\).
For every compact set \(K\Subset\Omega\), there is a finite constant
\(C\) such that, for all sufficiently large \(L\), there is a compact set
\(E_L\subset K\) satisfying
\begin{equation}\label{eq:bilip-exception}
 |K\setminus E_L|
 +|T(K)\setminus T(E_L)|
 \leq CL^{-P}
\end{equation}
and
\begin{equation}\label{eq:bilip-on-good}
 L^{-1}|x-y|
 \leq |T(x)-T(y)|
 \leq L|x-y|
 \qquad(x,y\in E_L).
\end{equation}
The exponent \(P\) is optimal: for every \(0<\lambda<1\), there
exist a solution of \eqref{eq:measure} and a compact set
\(K\Subset\Omega\) such that every measurable \(E\subset K\) with
\(\operatorname{Lip}(T|_E)\leq L\) satisfies
\begin{equation}\label{eq:bilip-sharp}
 |K\setminus E|\geq cL^{-P}
\end{equation}
for all sufficiently large \(L\).
\end{corollary}

The weak endpoint in Theorem~\ref{thm:main} determines the exact
exceptional-set power in \eqref{eq:bilip-exception}, while
Theorem~\ref{thm:sharp} proves its optimality.

The gradient-map conclusion also fits into the broader Sobolev
theory for optimal transport. De Philippis and Figalli
\cite{DFStability} established strong local \(W^{2,1}\) stability
for zero-boundary Dirichlet solutions under uniform positive density
bounds and strong \(L^1_{\mathrm{loc}}\) convergence of the
densities. They also proved strong local \(W^{1,1}\) stability of
optimal transport maps. For the natural boundary condition arising
in optimal transport between bounded convex domains, Savin and Yu
\cite{SavinYu2020} established global
\(W^{2,1+\varepsilon}\) estimates up to the boundary when the density
is bounded above and below by positive constants.

For Monge--Amp\`ere type equations, De Philippis and Figalli
\cite{DFType} obtained local \(W^{2,1+\varepsilon}\) regularity for
strictly \(c\)-convex transport potentials under suitable structural
assumptions on the cost and positive upper and lower density bounds.
In the planar optimal transport setting, Figalli and Loeper
\cite{FigalliLoeper2009} proved \(C^1\) regularity under an upper
bound on the Monge--Amp\`ere measure and suitable structural and
convexity assumptions on the cost and domains. For general costs and
smooth positive densities, De Philippis and Figalli
\cite{DFPartialRegularity} proved that optimal transport maps are
smooth outside a closed singular set of measure zero. For positive
continuous densities, Chen and Figalli \cite{ChenFigalli2017}
subsequently established \(W^{2,p}_{\mathrm{loc}}\) regularity of
transport potentials for every finite \(p\geq1\), outside such a
singular set.

In quadratic-cost optimal transport, Brenier's theorem
\cite{Brenier1991} identifies the optimal map with the gradient of a
convex potential. Caffarelli's interior regularity theory
\cite{CaffarelliMaps1992} makes Corollary~\ref{cor:bilip} applicable
to transport between bounded convex domains with positive bounded
densities. More precisely, let
\(\mu_i=\rho_i\,dx\) be probability measures on bounded open convex
domains \(\Omega_i\), with
\[
 0<a\leq\rho_i\leq b,\qquad a<b.
\]
For every compact \(K\Subset\Omega_0\) and all sufficiently large
\(L\), there is a compact \(E_L\subset K\) such that the Brenier map
\(T=\nabla u\) satisfies \eqref{eq:bilip-on-good} on \(E_L\), and
\begin{equation}\label{eq:transport-mass}
 \mu_0(K\setminus E_L)
 =
 \mu_1\bigl(T(K)\setminus T(E_L)\bigr)
 \leq C_KL^{-(a+b)/(b-a)}.
\end{equation}
This exponent is optimal. The proofs and the transport formulation
are contained in Section~\ref{sec:minty}.

We briefly describe the proof. For a smooth solution with \(D^2u>0\),
the partial Legendre transform defines two positive functions
\[
 a=\frac1{u_{11}},\qquad
 c=\frac{\det D^2u}{u_{11}}.
\]
In the transformed coordinates, denoted by \((\xi,\eta)\), they
satisfy
\[
 c_{\xi\xi}+a_{\eta\eta}=0,\qquad
 \lambda\leq\frac ca\leq\lambda^{-1}.
\]
We first study a periodic model on \(Q=(0,1)^2\). Periodicity permits
an exact potential representation and removes boundary terms from
integration by parts. Multiplying \(a\) and \(c\) by the same positive
constant preserves both the equation and the ratio bound, so we
normalize
\[
 \int_Qa=1.
\]
The main observation is the quadratic truncation inequality
\[
 \int_Q(a-s)_+(c-s)_+
 \geq\int_Q(a-s)(c-s),
 \qquad 0<s<1.
\]
To prove it, we adapt the one-coordinate envelope argument of
Sz\'ekelyhidi \cite[Section~3]{Szekelyhidi}. We take the convex
envelope of a suitable potential in one coordinate and show that this
operation does not increase the \(L^2\) variation of its slope in the
other coordinate. Combining the resulting truncation inequality with
the ratio bound, and varying the truncation level, we obtain a
small-value estimate for \(\min(a,c)\) with the precise exponent
\(P-1\).

For localization, we adapt the adjoint-solution comparison used by
Wang and Zhou \cite{WangZhou} in their study of the planar linearized
Monge--Amp\`ere equation. We extend only the coefficient \(c/a\)
periodically and construct a positive periodic adjoint solution for
the extended coefficient. The Harnack theorem of Bauman
\cite[Theorem~4.4]{Bauman} for normalized adjoint solutions controls
the ratio of this reference solution to the original local solution,
with constants depending only on ellipticity. This transfers the
periodic estimate without changing the exponent. Returning to the
original coordinates introduces the Jacobian factor \(a\), which
converts the small-value estimate into the weak-\(L^P\) Hessian
estimate. To pass to Alexandrov solutions, we rewrite this estimate
as a uniform bound on the Laplacian measures of measurable sets.
This bound passes to the limit under weak convergence and excludes a
singular part of the distributional Hessian.

For sharpness, we follow the iterative rank-one splitting construction
of Faraco \cite{Faraco2003} and the convex-integration method of
Astala, Faraco and Sz\'ekelyhidi \cite{AFS2008}. At each stage, two
successive rank-one splittings are realized by smooth Hessian
oscillations with arbitrarily small \(C^1\) perturbations, while
positivity and the determinant bounds are preserved. The scales and
splitting proportions are chosen so that the retained area is
comparable to the inverse \(P\)-th power of the Hessian scale.
The relative area errors are summable, so this relation persists
through the iteration. The trace bounds then establish the two-sided
estimate \eqref{eq:two-sided} for the resulting \(C^1\) limit.

Section~\ref{sec:periodic} proves the periodic truncation and
small-value estimates. Section~\ref{sec:local} establishes the local
adjoint estimate. Section~\ref{sec:mainproof} proves
Theorem~\ref{thm:main} and Corollary~\ref{cor:log}.
Section~\ref{sec:sharpness} completes the sharpness construction.
Section~\ref{sec:minty} proves Corollaries~\ref{cor:minty}
and~\ref{cor:bilip}, including the optimal transport application.

\section{A periodic truncation estimate}\label{sec:periodic}
We begin on the unit square $Q=(0,1)^2$, with all periodic functions
extended to $\R^2$. We write \(t_+=\max\{t,0\}\) and \(t_-=\max\{-t,0\}\). Positive constants denoted by $c$ and $C$ may change from line to line.

The partial Legendre transform used in
Section~\ref{sec:mainproof} motivates the following periodic problem.
In the transformed coordinates $(\xi,\eta)=(u_1(x_1,x_2),x_2)$, the two positive functions
\[
 a=\frac1{u_{11}},
 \qquad
 c=\frac{\det D^2u}{u_{11}}
\]
satisfy
\[
 c_{\xi\xi}+a_{\eta\eta}=0,
 \qquad
 \lambda\leq\frac ca\leq\lambda^{-1}.
\]
We first isolate the truncation inequality that depends only on the
differential relation; the ratio bound will be used afterward to
derive the critical lower tail.
\begin{lemma}\label{lem:truncation}
Suppose that \(a\) and \(c\) are smooth, positive, periodic functions satisfying
\[
 c_{xx}+a_{yy}=0,\qquad \int_Q a=1.
\]
Then for every $0<s<1$, we have
\begin{equation}\label{eq:truncation}
 \int_Q(a-s)_+(c-s)_+\geq\int_Q(a-s)(c-s).
\end{equation}
\end{lemma}

\begin{proof}
Let $\bar c=\int_Qc$, and solve the periodic Poisson equation
\[
 \Delta\chi=(a-1)-(c-\bar c),\qquad \int_Q\chi=0.
\]
The mean-zero periodic function $\chi_{xx}-(a-1)$ is harmonic and then vanishes. So
\begin{equation}\label{eq:potential}
 a=1+\chi_{xx},\qquad c=\bar c-\chi_{yy}.
\end{equation}
Fix $s\in(0,1)$. For each $y$, let $\psi(\cdot,y)$ be the greatest convex
function on $\R$ below
\[
 F(x,y)=\frac{1-s}{2}x^2+\chi(x,y).
\]
Quadratic growth makes this envelope finite. Translation of its convex
minorants shows that
\[
 \psi(x+1,y)=\psi(x,y)+(1-s)x+\frac{1-s}{2},\qquad \psi(x,y+1)=\psi(x,y).
\]
Consequently, if $r=\psi_x$ and $\omega=F-\psi$, then
$r-(1-s)x$ and $\omega$ are periodic.

We recall the elementary properties needed below. On every component of
$\{\psi<F\}$ in a horizontal line, $\psi$ is affine: otherwise replacing a short
piece by its chord would contradict maximality. Such components are bounded
by quadratic growth. At a contact point, \(\psi=F\) and \(\psi\leq F\) nearby. Convexity then
forces both one-sided slopes of \(\psi\) to equal \(F_x\). Thus $r$ is continuous, equals
$F_x$ on the contact set, and is constant on each remaining interval.
Comparison of contact endpoints also shows that $r$ is Lipschitz in $x$,
with Lipschitz constant at most $\|F_{xx}\|_\infty$.
On the contact set, the Lipschitz functions $r$ and $F_x$ have the same
derivative almost everywhere. Hence
\begin{equation}\label{eq:contact}
 0\leq r_x\leq(a-s)_+,\qquad \omega r_x=0
 \quad\hbox{almost everywhere}.
\end{equation}

The variation of $r$ in $y$ is controlled without taking a second derivative
of $\psi$ in that direction. Compare two horizontal lines, and denote the
corresponding functions by $f_i=F_x(\cdot,y_i)$, $r_i$ and $\omega_i$.
Both $f_i-r_i$ and $r_j-r_i$ are periodic. Integration by parts and
\eqref{eq:contact} imply
\[
 \int_0^1(f_i-r_i)(r_j-r_i)= \int_0^1 (\omega_i)_x(r_j - r_i)
 =-\int_0^1\omega_i(r_j)_x\leq0.
\]
Adding the inequalities for $(i,j)=(1,2),(2,1)$ and applying
Cauchy--Schwarz on $(0,1)$ yields
\begin{equation}\label{eq:contraction}
 \|r_1-r_2\|_2^2
 \leq\int_0^1(f_1-f_2)(r_1-r_2)
 \leq\|f_1-f_2\|_2\|r_1-r_2\|_2.
\end{equation}
The convex envelope changes by at most the uniform change in the
obstacle. Hence $\psi$ is continuous and $r$ is measurable.
Since the affine parts cancel, we may apply \eqref{eq:contraction}
with $y_2=y_1+\varepsilon$, divide by $\varepsilon^2$ and integrate
over one period in $y_1$. The difference-quotient characterization
of $H^1$ then yields
\begin{equation}\label{eq:transverse}
 r-(1-s)x\in H^1_{\rm per}(Q),\qquad
 \int_Q r_y^2\leq\int_Q\chi_{xy}^2.
\end{equation}

We now integrate by parts. From
\eqref{eq:potential}, periodicity and \eqref{eq:transverse},
\begin{align*}
 \int_Q(a-s)(c-s)
 &=(1-s)(\bar c-s)-\int_Q\chi_{xy}^2\\
 &\leq(1-s)(\bar c-s)-\int_Qr_y\chi_{xy}\\
 &=\int_Qr_x(\bar c -s) - \int_Qr_x \chi_{yy}   \\
 &=\int_Qr_x(c-s)
 \leq\int_Q(a-s)_+(c-s)_+.
\end{align*}
The first equality is due to $\int_Q \chi_{xx} =0$ and $\int_Q \chi_{yy}=0$. The first inequality is due to $\int _Q r_y \chi _{xy} \leq (\int_Q r_y^2)^\frac{1}{2}(\int_Q \chi_{xy}^2)^\frac{1}{2} \leq \int _Q \chi_{xy}^2$. The last inequality also holds where $c<s$, because $r_x\geq0$ there.
This proves \eqref{eq:truncation}.
\end{proof}
We now combine \eqref{eq:truncation} with
\[
 \lambda\leq\frac ca\leq\lambda^{-1}.
\]
This converts the quadratic inequality into a differential inequality
for the lower tail of \(\min(a,c)\).
\begin{proposition}\label{prop:periodic}
Under the hypotheses of Lemma~\ref{lem:truncation}, suppose also
$\lambda\leq c/a\leq\lambda^{-1}$. Then
\begin{equation}\label{eq:periodic-tail}
 |\{\min(a,c)<t\}\cap Q|\leq C_\lambda t^{P-1},\qquad
 \int_{Q\cap\{a<t\}}a\leq C_\lambda t^P\quad(t>0).
\end{equation}
For a square of arbitrary side length and $\fint_Qa=1$, the right-hand
sides are multiplied by $|Q|$.
\end{proposition}

\begin{proof}
Let $v=\min(a,c)$ and $I(s)=\int_Q(s-v)_+^2$. The difference between the
two integrands in \eqref{eq:truncation} is
$(s-v)_+(\max(a,c)-s)$. Since $\max(a,c)\leq v/\lambda$,
\[
 0\leq\int_Q(s-v)_+(v-\lambda s)
 =\frac{1-\lambda}{2}sI'(s)-I(s),\qquad 0<s<1.
\]
Hence $sI'(s)\geq(P+1)I(s)$ and $I(s)/s^{P+1}$ is nondecreasing.
Since $I(1)\leq1$, continuity at $1$ implies $I(s)\leq s^{P+1}$.
For $0<t\leq1/2$ we have
\[
 t^2|\{v<t\}|\leq I(2t)\leq2^{P+1}t^{P+1}.
\]
For $t>1/2$ the estimate follows from $|Q|=1$.\\
To show the second inequality, since
\(\{a<t\}\subset\{v<t\}\), we have
\[
 \int_{Q\cap\{a<t\}}a\leq \frac1\lambda\int_{Q\cap\{v<t\}}v=\frac1\lambda\int_0^t
   |\{x\in Q:s<v(x)<t\}|\,ds\leq \frac{t}{\lambda}|\{v<t\}\cap Q|
 \leq C_\lambda t^P.
\]
\end{proof}

\section{A local adjoint estimate}\label{sec:local}
In this section, we transfer the periodic estimate to a local positive adjoint solution
using a positive periodic reference and a ratio Harnack estimate.
\begin{lemma}\label{lem:reference}
If $g$ is smooth and periodic on a square $Q$, with
$\lambda\leq g\leq\lambda^{-1}$, there exists a positive smooth periodic
$m$ such that
\begin{equation}\label{eq:reference}
 (gm)_{xx}+m_{yy}=0,\qquad \fint_Qm=1.
\end{equation}
Moreover, $|\{m<t\}\cap Q|\leq C_\lambda|Q|t^{P-1}$ for every $t>0$.
\end{lemma}

\begin{proof}
Regard
\[
    L=g\partial_{xx}+\partial_{yy}
\]
as an operator from \(H^2(\mathbb T^2)\) to \(L^2(\mathbb T^2)\).
For
\[
    g_t=(1-t)+tg,
    \qquad
    L_t=g_t\partial_{xx}+\partial_{yy},
    \qquad 0\le t\le1,
\]
we have
\[
    \lambda\le g_t\le\lambda^{-1}.
\]
Thus \(L_t\) is a continuous family of elliptic Fredholm operators,
and its index is independent of \(t\). Thus 
$$
\operatorname{ind} L=\operatorname{ind}\Delta=0.
$$
Moreover, the maximum principle implies \(\ker L=\mathbb R\). Therefore
\[
    \dim\ker L^*
    =
    \operatorname{codim}\operatorname{Ran}L
    =1.
\]
Elliptic regularity shows that every element of \(\ker L^*\) is
smooth. Choose a nonzero \(m\in\ker L^*\). Since
\(\ker L^*\) is one-dimensional, \(m\) spans this kernel. The function
\(m\) cannot change sign: otherwise one could choose a smooth
function \(F>0\) such that
\[
 \int_Q Fm=0.
\]
The Fredholm alternative would then yield a periodic solution of
\(Lv=F>0\). At a maximum point \(x_0\) of \(v\), however,
\(Lv(x_0)\leq0\), a contradiction. Replacing \(m\) by
\(-m\) if necessary, we may assume that \(m\geq0\). To check strict positivity, expand
\[
 L^*m=g m_{xx}+m_{yy}+2g_xm_x+g_{xx}m=0.
\]
Then
\[
 \bigl(g\partial_{xx}+\partial_{yy}
       +2g_x\partial_x-(g_{xx})_-\bigr)m
 =-(g_{xx})_+m\leq0.
\]
The strong minimum principle implies $m>0$. Normalize its mean and apply
Proposition~\ref{prop:periodic} to $(m,gm)$.
\end{proof}

We next compare the periodic reference with a local positive adjoint
solution. The required comparison follows from Bauman's Harnack theorem
for normalized adjoint solutions.

\begin{lemma}\label{lem:ratio-harnack}
Let
\[
 L=g\partial_{xx}+\partial_{yy}
 \qquad\text{in }B_{8r},
\]
where $g$ is smooth and
\[
 \lambda\leq g\leq\lambda^{-1}.
\]
If $a,m>0$ are smooth in $B_{4r}$ and satisfy
\[
 L^*a=(ga)_{xx}+a_{yy}=0,
 \qquad
 L^*m=(gm)_{xx}+m_{yy}=0,
\]
then
\begin{equation}\label{eq:ratio-harnack}
 \sup_{B_r}\frac{a}{m}
 \leq C_\lambda\inf_{B_r}\frac{a}{m}.
\end{equation}
The constant depends only on ellipticity, not on derivatives or a
modulus of continuity of $g$.
\end{lemma}

\begin{proof}
Choose a pole in
\[
 B_{8r}\setminus\overline{B_{6r}},
\]
and let $G$ be the corresponding positive Green function of $L$ in
$B_{8r}$, viewed as a function of its second variable. Since the pole
lies outside $B_{6r}$,
\[
 L^*G=0\qquad\text{in }B_{6r}.
\]
Hence
\[
 \frac{a}{G}
 \qquad\text{and}\qquad
 \frac{m}{G}
\]
are positive normalized adjoint solutions in $B_{4r}$.

Applying Bauman's Harnack theorem
\cite[Theorem~4.4]{Bauman} to the positive normalized adjoint
solutions \(a/G\) and \(m/G\) in \(B_{4r}\), we obtain, for every
\(x,y\in B_r\),
\[
 \frac{a(x)}{G(x)}
 \leq C_\lambda\frac{a(y)}{G(y)},
 \qquad
 \frac{m(y)}{G(y)}
 \leq C_\lambda\frac{m(x)}{G(x)}.
\]
Therefore,
\[
 \frac{a(x)}{m(x)}
 =
 \frac{a(x)/G(x)}{m(x)/G(x)}
 \leq
 C_\lambda\frac{a(y)/G(y)}{m(y)/G(y)}
 =
 C_\lambda\frac{a(y)}{m(y)},
\]
where the constant has been enlarged. Taking the supremum in $x$ and
the infimum in $y$ proves \eqref{eq:ratio-harnack}.
\end{proof}

Combining the periodic lower-tail estimate with
Lemma~\ref{lem:ratio-harnack}, we obtain the corresponding local estimate.

\begin{proposition}\label{prop:local}
Suppose $a,g$ are smooth on $B_{8r}$,
\[
 a>0,\qquad (ga)_{xx}+a_{yy}=0,\qquad
 \lambda\leq g\leq\lambda^{-1}.
\]
With $A=\fint_{B_r}a$, one has, for every $t>0$,
\begin{equation}\label{eq:local-tail}
 \int_{B_r\cap\{a<t\}}a
 \leq C_\lambda |B_r|A^{1-P}t^P.
\end{equation}
\end{proposition}

\begin{proof}
After translation, the balls are centered at zero. Choose \(0\leq\zeta\leq1\) such that
\[
 \zeta=1\quad\text{on }B_{4r},
 \qquad
 \supp\zeta\subset B_{6r},
\]
and define on \(Q_r=(-8r,8r)^2\)
\[
 \widehat g=
 \begin{cases}
  1+\zeta(g-1),&\text{in }B_{6r},\\
  1,&\text{in }Q_r\setminus B_{6r}.
 \end{cases}
\]
Then \(\widehat g\equiv1\) near \(\partial Q_r\), so it extends smoothly
and periodically. Moreover,
\[
 \lambda\leq\widehat g\leq\lambda^{-1},
 \qquad
 \widehat g=g\quad\text{on }B_{4r}.
\] Let $m$ be the reference in
Lemma~\ref{lem:reference}. On $B_{4r}$, $a$ and $m$ solve the same adjoint
equation. By \eqref{eq:ratio-harnack},
\[
 A|B_r|=\int_{B_r}a
 \leq C_\lambda\inf_{B_r}(a/m)\int_{B_r}m
 \leq C_\lambda\inf_{B_r}(a/m)|Q_r|.
\]
Since $|Q_r|/|B_r|$ is fixed, $a\geq c_\lambda A m$ on $B_r$. Thus
\[
 |\{a<t\}\cap B_r|
 \leq |\{m<C_\lambda t/A\}\cap Q_r|
 \leq C_\lambda|B_r|(t/A)^{P-1}.
\]
Since \(a<t\) on this set,
\[
 \int_{B_r\cap\{a<t\}}a
 \leq t\,|\{a<t\}\cap B_r|,
\]
and \eqref{eq:local-tail} follows.
\end{proof}

\section{Proof of Theorem~\ref{thm:main}}\label{sec:mainproof}
We first prove the estimate for smooth strictly convex solutions, and then pass to Alexandrov solutions and remove the normalization on
compact sections.
\subsection{The smooth case}
For a smooth solution with $D^2u>0$, consider the partial Legendre transform
\[
 (\xi,\eta)=\Phi(x_1,x_2)=(u_1(x_1,x_2),x_2),\qquad
 w(\xi,\eta)=x_1\xi-u(x_1,x_2).
\]
The map is one-to-one on a convex domain, because $u_1$ is strictly
increasing on each horizontal segment. Its positive Jacobian is $u_{11}$.
Direct differentiation shows that
\begin{equation}\label{eq:legendre}
 a=w_{\xi\xi}=\frac1{u_{11}},\qquad
 c=-w_{\eta\eta}=\frac{\det D^2u}{u_{11}},\qquad
 c_{\xi\xi}+a_{\eta\eta}=0,\qquad dx=a\,d\xi\,d\eta.
\end{equation}
Thus $g=c/a$ remains in $[\lambda,\lambda^{-1}]$. For every ball \(B_r\) in the transformed coordinates such that
\(B_{8r}\subset\Phi(\Omega)\),
Proposition~\ref{prop:local} implies
\begin{equation}\label{eq:physical-tail}
 |\{u_{11}>T\}\cap\Phi^{-1}(B_r)|
 =\int_{B_r\cap\{a<T^{-1}\}}a
 \leq C_\lambda|B_r|A^{1-P}T^{-P},
 \qquad A=\fint_{B_r}a.
\end{equation}

It remains to choose the transformed balls uniformly and to bound the
average \(A\) in \eqref{eq:physical-tail} from below. For the normalized
class in Theorem~\ref{thm:main}, quadratic comparison and the Alexandrov
maximum principle imply
\[
 \frac{\sqrt\lambda}{2}\le h\le\frac{2}{\sqrt\lambda},
 \qquad
 |u(x)|^2\le C_\lambda\dist(x,\partial\Omega).
\]
Hence every fixed lower sublevel stays a uniform distance from
\(\partial\Omega\). Moreover, the normalized family is compact under local uniform convergence.
Indeed, every sequence has a locally uniformly convergent subsequence,
and Alexandrov stability preserves the measure bounds. The planar
strict convexity and interior \(C^{1,\alpha}\) theory imply that the
limit is \(C^1\) on fixed lower sublevels. Uniform convergence of
convex functions to a \(C^1\) limit implies uniform convergence of
their gradients on compact subsets. A contradiction argument
therefore yields a common modulus of continuity for \(Du\) on fixed
lower sublevels; see
\cite[Sections~2.1 and~2.3--2.4]{DFSurvey}.

Set
\[
 \theta_1=\frac{1+\theta}{2}.
\]
Since \(\Phi\) is a diffeomorphism near
\(\overline{S_{\theta_1}}\),
\[
 \partial\Phi(S_{\theta_1})=\Phi(\partial S_{\theta_1}),
 \qquad
 \operatorname{diam}\Phi(S_{\theta_1})\le C(\lambda,\theta).
\]
There is \(r=r(\lambda,\theta)>0\) such that
\begin{equation}\label{eq:separation}
 \dist\bigl(\Phi(S_\theta),
            \partial\Phi(S_{\theta_1})\bigr)\ge 8r.
\end{equation}
Otherwise, compactness would produce limit points
\(x\in\overline{S_\theta}\) and
\(y\in\partial S_{\theta_1}\) satisfying
\[
 x_2=y_2,\qquad u_1(x)=u_1(y),
\]
while
\[
 u(y)-u(x)
 \ge(\theta_1-\theta)h
 \ge\frac{\sqrt\lambda}{2}(\theta_1-\theta)>0.
\]
The restriction of \(u\) to the horizontal segment joining \(x\) and
\(y\) is convex. Its derivative has the same value at both endpoints,
so it is constant on the segment. This contradicts planar strict
convexity and proves \eqref{eq:separation}.

By the common modulus of continuity of \(Du\) on
\(\overline{S_{\theta_1}}\), choose
\(\rho=\rho(\lambda,\theta)>0\) such that
\[
 |Du(p)-Du(q)|+|p-q|<r
\]
whenever
\[
 p,q\in\overline{S_{\theta_1}},
 \qquad |p-q|<\rho.
\]

Fix \(x\in S_\theta\) and set \(z=\Phi(x)\). We first claim that
\[
 B_\rho(x)\subset S_{\theta_1}.
\]
Otherwise, since \(S_{\theta_1}\) is convex and
\(x\in S_{\theta_1}\), there is
\[
 y\in\partial S_{\theta_1}\cap B_\rho(x).
\]
Since
\[
 \Phi(y)\in\partial\Phi(S_{\theta_1}),
\]
equation \eqref{eq:separation} implies
\[
 8r
 \le |\Phi(x)-\Phi(y)|
 \le |Du(x)-Du(y)|+|x-y|
 <r,
\]
a contradiction.

For every \(y\in B_\rho(x)\), the first inclusion shows that
\(x,y\in S_{\theta_1}\). Hence
\[
 |\Phi(y)-\Phi(x)|
 \le |Du(y)-Du(x)|+|y-x|
 <r,
\]
and therefore
\[
 \Phi(B_\rho(x))\subset B_r(z).
\]
By \eqref{eq:legendre},
\[
 \int_{B_r(z)}a\,d\xi\,d\eta
 =|\Phi^{-1}(B_r(z))|
 \ge |B_\rho(x)|.
\]
Consequently,
\[
 A=\fint_{B_r(z)}a
 \ge \frac{|B_\rho|}{|B_r|}
 \ge c(\lambda,\theta)>0.
\]

Since \(\Phi(S_\theta)\) has uniformly bounded diameter, it is covered
by at most \(C(\lambda,\theta)\) such balls. Applying \eqref{eq:physical-tail} to each ball in this covering yields
\[
 |\{u_{11}>T\}\cap S_\theta|
 \le C(\lambda,\theta)T^{-P}.
\]
Interchanging \(x_1\) and \(x_2\) yields the same estimate for
\(u_{22}\). Finally,
\[
 \{\Delta u>T\}
 \subset
 \{u_{11}>T/2\}\cup\{u_{22}>T/2\},
\]
because \(D^2u\ge0\). This proves the smooth estimate.

\subsection{Passing to Alexandrov solutions}

We first treat the normalized zero-boundary case. Write
\[
    \mu_u=f\,dx,
    \qquad
    \lambda\le f\le\lambda^{-1}.
\]
Extend \(f\) by \(1\) outside \(\Omega\) and mollify it. The resulting
smooth functions \(f_j\) remain in \([\lambda,\lambda^{-1}]\).
Choose smooth uniformly convex domains \(\Omega_j\) converging to
\(\Omega\) in Hausdorff distance, with
\[
    B_{1/2}\subset\Omega_j\subset B_3,
\]
and let \(u_j\) solve
\[
    \det D^2u_j=f_j\quad\text{in }\Omega_j,
    \qquad
    u_j=0\quad\text{on }\partial\Omega_j.
\]
Extend \(u_j\) by zero outside \(\Omega_j\). Stability of Alexandrov
solutions then implies
\[
    u_j\longrightarrow u\quad\text{uniformly},
    \qquad
    h_j:=-\min_{\Omega_j}u_j\longrightarrow h;
\]
see \cite[Theorems~2.8 and~2.11]{DFSurvey}.

Set
\[
    \theta_1=\frac{1+\theta}{2}.
\]
Uniform convergence provides an open set
\[
    \overline{S_\theta}\subset U\Subset\Omega
\]
such that, for all large \(j\),
\[
    U\subset\{u_j<-(1-\theta_1)h_j\}.
\]
Since the domains have fixed inner and outer ball bounds, the smooth
estimate applies uniformly on \(U\):
\[
    |\{\Delta u_j>T\}\cap U|
    \le C(\lambda,\theta)T^{-P}.
\]

To pass the endpoint estimate to the limit, we rewrite the weak tail
as a uniform estimate for the positive measures
\(\Delta u_j\,dx\). Let
\[
    \nu_j=\Delta u_j\,dx.
\]
For every measurable \(E\subset U\), we have
\begin{equation}\label{eq:set-bound}
    \nu_j(E)\le
      \int_0^\infty\min\{|E|,CT^{-P}\}\,dT=\frac{P}{P-1}C^{1/P}|E|^{1-1/P}.
\end{equation}
Local uniform convergence implies
\[
    \int\varphi\,d\nu_j
    =
    \int u_j\Delta\varphi\,dx
    \longrightarrow
    \int u\Delta\varphi\,dx
\]
for every \(\varphi\in C_c^\infty(U)\). Hence
\[
    \nu_j\rightharpoonup\nu:=\Delta u
\]
locally as positive Radon measures. For every open set \(O\Subset U\),
\[
 \nu(O)
 \leq\liminf_{j\to\infty}\nu_j(O)
 \leq C|O|^{1-1/P}.
\]
Outer
regularity then shows
\begin{equation}\label{eq:limit-set-bound}
 \nu(E)\le C|E|^{1-1/P}
\end{equation}
for every measurable \(E\Subset U\). In particular,
\[
    \nu\ll dx.
\]

The distributional Hessian of a convex function is a positive
semidefinite matrix-valued measure. Thus
\[
    |D_{ij}u|
    \le \frac12\bigl(D_{ii}u+D_{jj}u\bigr)
    \le \Delta u
\]
as measures, and therefore
\[
    u\in W^{2,1}_{\mathrm{loc}}(U).
\]
Identifying \(\Delta u\) with its density, apply
\eqref{eq:limit-set-bound} to
\[
 E_T=\{\Delta u>T\}\cap S_\theta.
\]
Then
\[
    T|E_T|
    \le \int_{E_T}\Delta u\,dx
    \le C|E_T|^{1-1/P}.
\]
Consequently,
\[
    |\{\Delta u>T\}\cap S_\theta|
    \le C(\lambda,\theta)T^{-P}.
\]

It remains to remove the normalized-domain and zero-boundary
assumptions. For a general local solution, choose a compact section
\[
    \mathcal S=\{u<\ell\}\Subset\Omega,
\]
where \(\ell\) is affine. Such sections form a neighborhood basis by
planar strict convexity. Let
\[
    y=Ax+b
\]
normalize \(\mathcal S\), and define
\[
    \widetilde u(y)
    =
    |\det A|\,
    (u-\ell)\bigl(A^{-1}(y-b)\bigr).
\]
For every Borel set \(E\),
\[
    \mu_{\widetilde u}(E)
    =
    |\det A|\,
    \mu_u\bigl(A^{-1}(E-b)\bigr),
    \qquad
    |A^{-1}(E-b)|
    =
    |\det A|^{-1}|E|.
\]
Hence
\[
    \lambda\,dy
    \le \mu_{\widetilde u}
    \le \lambda^{-1}\,dy.
\]
At almost every point of twice differentiability,
\[
 D_y^2\widetilde u(y)
 =
 |\det A|\,
 A^{-T}D_x^2u(x)A^{-1},
 \qquad
 x=A^{-1}(y-b).
\]
Equivalently,
\[
 D_x^2u(x)
 =
 |\det A|^{-1}
 A^T D_y^2\widetilde u(y)A.
\]
Since \(A\) is fixed on each section, this linear change of variables
and its constant Jacobian preserve local weak \(L^P\) membership, up
to a constant depending on \(A\). A smaller section corresponds to a
fixed lower sublevel of the normalized section. Since such sections
form a neighborhood basis, the normalized estimate transfers to every
compact subset of \(\Omega\).

For \(1\leq p<P\), the layer-cake formula and the local weak
\(L^P\) estimate imply, for every \(K\Subset\Omega\),
\[
 \int_K(\Delta u)^p\,dx
 =
 p\int_0^\infty
 T^{p-1}|\{\Delta u>T\}\cap K|\,dT
 <\infty.
\]
Thus
\[
 u\in W^{2,p}_{\mathrm{loc}}(\Omega)
 \qquad\text{for every }1\leq p<P,
\]
and Theorem~\ref{thm:main} follows.

We next derive the logarithmic consequence of the weak endpoint
estimate.
\begin{proof}[Proof of Corollary~\ref{cor:log}]
Integrating \eqref{eq:weak} proves
\eqref{eq:logarithmic}: for large \(T\), the resulting integrand is
bounded by
\[
 \frac{C}{T(\log T)^\gamma}.
\]
The lower bound in \eqref{eq:two-sided}, proved in
Section~\ref{sec:sharpness}, shows that the condition
\(\gamma>1\) is optimal.
\end{proof}

\section{Sharpness}\label{sec:sharpness}
In this section, we prove Theorem~\ref{thm:sharp} by successive
rank-one splittings, following Faraco's iterative splitting method
\cite{Faraco2003} and its adaptation to symmetric Hessians by Astala,
Faraco and Sz\'ekelyhidi \cite{AFS2008}. We include the details needed to preserve
the determinant bounds and obtain the tail at every sufficiently large level.

\subsection{Two elementary matrix splittings}
For $t\geq1$, let
$$
A_t=\diag(t,t^{-1}).
$$
Choose
$\lambda<\tau<1$ and $1<s<\min\{2,\tau^{-1}\}$. First split $A_t$
in the first diagonal entry between
\[
 L_t=\diag(\tau t,t^{-1}),\qquad B_t=\diag(st,t^{-1}),
\]
and then split $B_t$ in the second diagonal entry between
\[
 U_t=\diag(st,(\tau st)^{-1}),\qquad A_{st}.
\]
Each difference has rank one. The combined identity is
\begin{equation}\label{eq:split}
 A_t=\ell L_t+\tau\ell U_t+\vartheta A_{st},\qquad
 \ell=\frac{s-1}{s-\tau},\qquad
 \vartheta=\frac{1-\tau s}{s-\tau}.
\end{equation}
The weights are positive and sum to one. Both matrix segments are positive
definite, with determinant in $[\tau,\tau^{-1}]$, and hence lie strictly
inside $\{M>0:\lambda<\det M<\lambda^{-1}\}$.

\begin{lemma}\label{lem:splitting-step}
Let
\[
 t\geq1,\qquad
 \lambda<\tau<1,\qquad
 1<s<\min\{2,\tau^{-1}\},
\]
and let \(\vartheta\) be defined in \eqref{eq:split}. Suppose that
\(E\) is a finite union of open squares with pairwise disjoint
closures, \(u\) is smooth near \(\overline E\), and
\[
 D^2u=A_t\qquad\text{in }E.
\]
For every \(0<\varepsilon<1\) and \(\varepsilon_0>0\), there exist
\[
 \varphi\in C_c^\infty(E)
\]
and a finite union \(E'\Subset E\) of open squares with pairwise
disjoint closures such that, with
\[
 u_+=u+\varphi,
\]
one has
\[
 \|\varphi\|_{C^1(\mathbb R^2)}<\varepsilon_0,
\]
\[
 D^2u_+>0,\qquad
 \lambda<\det D^2u_+<\lambda^{-1}
 \quad\text{in }E,
\]
\[
 D^2u_+=A_{st}\quad\text{on }E',
 \qquad
 (1-\varepsilon)\vartheta|E|
 \leq |E'|
 \leq(1+\varepsilon)\vartheta|E|,
\]
and
\[
 \frac{\lambda}{2}t
 \leq\tr D^2u_+
 \leq C_\lambda st
 \qquad\text{in }E.
\]
\end{lemma}

\begin{proof}
We first realize one rank-one splitting. Suppose
\[
 D^2u=M
\]
on an open square \(\mathcal Q\), and consider a splitting
\[
 M=\alpha M^-+(1-\alpha)M^+,
 \qquad
 M^\pm=M+b_\pm e_i\otimes e_i,
\]
where
\[
 b_-<0<b_+,
 \qquad
 \alpha b_-+(1-\alpha)b_+=0.
\]
Thus only the \(i\)-th diagonal entry changes.

Let \(q\) be the \(1\)-periodic step function that equals \(b_-\)
on a proportion \(\alpha\) of each period and \(b_+\) on the
remaining proportion. Then
\[
 \int_0^1q=0.
\]
After periodic convolution with a nonnegative smooth kernel of small
support, \(q\) remains in \([b_-,b_+]\), still has zero mean, and
equals \(b_-\) or \(b_+\) outside transition intervals of arbitrarily
small total length. Two periodic integrations produce a smooth
\(1\)-periodic function \(h\) satisfying
\[
 h''=q.
\]
The constant in the first integration is chosen so that
\(\int_0^1h'=0\).

Choose \(0\leq\eta\leq1\), compactly supported in \(\mathcal Q\) and equal to
one outside a thin boundary strip. For large \(j\), set
\[
 \varphi_j(x)=j^{-2}\eta(x)h(jx_i).
\]
A direct calculation shows that $\|\varphi_j\|_{C^1}=O(j^{-1})$, and
\begin{equation}\label{eq:oscillation}
 D^2\varphi_j
 =\eta h''(jx_i)e_i\otimes e_i+j^{-1}h'(jx_i)
   \bigl(D\eta\otimes e_i+e_i\otimes D\eta\bigr)
   +j^{-2}h(jx_i)D^2\eta.
\end{equation}
In particular,
\[
 D^2\varphi_j
 =
 \eta h''(jx_i)e_i\otimes e_i+O(j^{-1})
\]
uniformly once \(\eta\) is fixed.

Since
\[
 0\leq\eta\leq1,
 \qquad
 b_-\leq h''\leq b_+,
\]
the matrix
\[
 M+\eta h''(jx_i)e_i\otimes e_i
\]
lies on the segment joining \(M^-\) and \(M^+\). For either splitting
used below, this segment is compactly contained in
\[
 \{N=N^T>0:\lambda<\det N<\lambda^{-1}\}.
\]
Hence, for sufficiently large \(j\),
\[
 D^2(u+\varphi_j)>0,
 \qquad
 \lambda<\det D^2(u+\varphi_j)<\lambda^{-1}
 \quad\text{in }\mathcal Q.
\]
Where \(\eta=1\) and \(h''=b_\pm\), the new Hessian is exactly
\(M^\pm\). Choose the boundary and transition regions with arbitrarily small
total area, and then take \(j\) large. Outside these exceptional
regions, the new Hessian equals \(M^-\) and \(M^+\) on sets whose
relative areas are arbitrarily close to \(\alpha\) and \(1-\alpha\),
respectively. At the same time,
\[
 \|\varphi_j\|_{C^1}\longrightarrow0.
\]

Apply this construction first to
\[
 A_t=\ell L_t+(1-\ell)B_t.
\]
The regions where the Hessian equals \(B_t\) contain, up to an arbitrarily small area
loss, finitely many squares with pairwise disjoint closures. Inside
those squares apply the second splitting
\[
 B_t
 =
 \frac{\tau(s-1)}{1-\tau}\,U_t
 +
 \frac{1-\tau s}{1-\tau}\,A_{st}.
\]
The total proportion of the \(A_{st}\)-phase is
\[
 (1-\ell)\frac{1-\tau s}{1-\tau}
 =
 \frac{1-\tau s}{s-\tau}
 =
 \vartheta.
\]
After the two splittings, the regions where the Hessian equals
\(L_t\) or \(U_t\) are left unchanged, while only the
\(A_{st}\)-regions are retained for the next stage.

Perform the construction independently in the finitely many squares
forming \(E\), and sum the resulting perturbations. By increasing the
frequencies, the sum \(\varphi\) may be chosen to satisfy
\[
 \varphi\in C_c^\infty(E),
 \qquad
 \|\varphi\|_{C^1(\mathbb R^2)}<\varepsilon_0.
\]
Up to an arbitrarily small area loss, the regions where
\(D^2u_+=A_{st}\) contain a finite union \(E'\Subset E\) of squares
with pairwise disjoint closures. The omitted boundary and transition
regions may be chosen so that
\begin{equation}\label{eq:step-area}
 D^2u_+=A_{st}\quad\text{on }E',
 \qquad
 (1-\varepsilon)\vartheta|E|
 \leq |E'|
 \leq(1+\varepsilon)\vartheta|E|.
\end{equation}
The perturbation vanishes near the boundary of every parent square,
so the local constructions fit together smoothly.

Every matrix \(N\) on the two splitting segments is positive
definite and satisfies
\[
 \tau\leq\det N\leq\tau^{-1},
 \qquad
 N_{11}\geq\tau t>\lambda t,
 \qquad
 \tr N\leq C_\lambda st.
\]
After reducing the transition errors,
\[
 D^2u_+>0,\qquad
 \lambda<\det D^2u_+<\lambda^{-1}
 \quad\text{in }E,
\]
and
\[
 \frac{\lambda}{2}t
 \leq\tr D^2u_+
 \leq C_\lambda st
 \qquad\text{in }E.
\]
This proves the lemma.
\end{proof}

\subsection{Choosing the scales and taking the limit}
If \(\tau>\lambda\) were fixed, the surviving area would decay with
the larger exponent
\[
 \frac{1+\tau}{1-\tau}>P.
\]
To reach the critical exponent \(P\), we let \(\tau\) approach
\(\lambda\) while the scale grows by factors \(s_k\downarrow1\).

Take an integer $k_0$, sufficiently large depending only on $\lambda$,
and for $k\geq k_0$ define
\begin{equation}\label{eq:scales}
 t_k=\frac{k}{k_0},\qquad s_k=\frac{k+1}{k},\qquad
 \tau_k=\lambda+\frac{1-\lambda}{k^2},\qquad
 \vartheta_k=\frac{1-\tau_ks_k}{s_k-\tau_k}.
\end{equation}
These parameters satisfy the conditions above and ensure the summable error
\begin{equation}\label{eq:product-asymptotic}
 -\log\vartheta_k=P\log s_k+O_\lambda(k^{-3}).
\end{equation}
For completeness, the exact derivative identity
\[
 -s\frac{d}{ds}\log\frac{1-\tau s}{s-\tau}
 -\frac{1+\tau}{1-\tau}
 =\frac{\tau(1+\tau)(s-1)^2}
 {(s-\tau)(1-\tau)(1-\tau s)}
\]
shows that the error relative to
$\frac{1+\tau}{1-\tau}\log s$ is $O_\lambda((s-1)^3)$ near
$(s,\tau)=(1,\lambda)$. Since
$\tau_k-\lambda=O_\lambda(k^{-2})$ and $\log s_k=O(k^{-1})$,
\eqref{eq:product-asymptotic} follows. Summing it and using
$\prod_{j=k_0}^{k-1}s_j=t_k$, we obtain
\begin{equation}\label{eq:product}
 c_\lambda t_k^{-P}\leq\prod_{j=k_0}^{k-1}\vartheta_j
 \leq C_\lambda t_k^{-P}.
\end{equation}
The quantity \(t_k\) is the Hessian scale after \(k-k_0\) continuation
steps, while
\[
 \prod_{j=k_0}^{k-1}\vartheta_j
\]
is the area fraction predicted by the exact matrix splittings.
Equation~\eqref{eq:product} identifies their critical relation.

Start with $u_{k_0}=|x|^2/2$ and $E_{k_0}=Q$. At stage $k$, apply
Lemma~\ref{lem:splitting-step} with \(E=E_k\), the parameters in
\eqref{eq:scales}, relative error $2^{-k-2}$, and
\(C^1\) tolerances whose sum is as small as prescribed. Then
\[
 u_{k+1}=u_k+\varphi_k,\qquad \supp\varphi_k\Subset E_k,
 \qquad E_{k+1}\Subset E_k,\qquad D^2u_{k+1}=A_{t_{k+1}} \hbox{on }E_{k+1}.
\]
Because $\sum_{k\geq k_0}2^{-k-2}<\infty$, \eqref{eq:step-area} and
\eqref{eq:product} imply
\begin{equation}\label{eq:marked-area}
 c_\lambda|Q|t_k^{-P}\leq |E_k|\leq C_\lambda|Q|t_k^{-P}.
\end{equation}
The functions $u_k$ converge in $C^1$ to a convex function $u$, equal to
$|x|^2/2$ outside $Q$, with the prescribed $C^1$ distance from this
quadratic polynomial. The compact containments imply that
$E_\infty=\bigcap_kE_k=\bigcap_k\overline{E_k}$ is compact,
and \eqref{eq:marked-area} shows that it has area zero.
Outside this set, all sufficiently late perturbations vanish on a
neighborhood of each point, so $u$ is smooth there.

We next verify that the limit remains in the admissible
Monge--Ampère class and that its distributional Hessian has no
singular part. Local uniform convergence of convex functions preserves the full
Monge--Amp\`ere measure bounds \cite[Proposition~2.4]{DFSurvey}.
Hence $u$ satisfies \eqref{eq:measure}, and planar strict convexity
implies that $u$ is strictly convex. By Theorem~\ref{thm:main}, the distributional Hessian of \(u\) has no
singular part. Its density therefore agrees almost everywhere with
the classical Hessian on \(\R^2\setminus E_\infty\).

The trace bounds at each stage imply, for the pointwise limiting Hessian
outside $E_\infty$,
\begin{equation}\label{eq:trace-layers}
 \Delta u\geq\frac\lambda2t_k\quad\hbox{almost everywhere on }E_k,
 \qquad \Delta u\leq C_\lambda t_k\quad\hbox{on }Q\setminus E_k.
\end{equation}
Indeed, if a point belongs to \(E_k\) and lies in
\(E_j\setminus E_{j+1}\) for some \(j\geq k\), then its final Hessian
satisfies the lower bound at the larger scale \(t_j\geq t_k\).
Points outside \(E_k\) are not modified after stage \(k\). For the upper tail, choose $k$ so that
$C_\lambda t_k\leq T<C_\lambda t_{k+1}$; then
$\{\Delta u>T\}\cap Q\subset E_k$ up to a null set.
For the lower tail, choose the first $k$ such that
$\lambda t_k/2>T$; then $E_k\subset\{\Delta u>T\}$ up to a null set.
Since $t_{k+1}/t_k\leq2$, \eqref{eq:marked-area} yields
\eqref{eq:two-sided} in both cases.

Finally, the lower bound in \eqref{eq:two-sided} implies that
\[
 \int_{T_\lambda}^\infty
 \bigl[T^P|\{\Delta u>T\}\cap Q|\bigr]^{b/P}\,\frac{dT}{T}
\]
diverges for every finite $b>0$. Up to a constant, this integral is the \(b\)-th power of the
\(L^{P,b}(Q)\) quasi-norm of \(\Delta u\). Since the Hessian norm and trace
are comparable, Theorem~\ref{thm:sharp} follows. The same two-sided tail also shows that
\[
 \int_Q\frac{(\Delta u)^P}{[\log(e+\Delta u)]^\gamma}\,dx<\infty
 \quad\Longleftrightarrow\quad \gamma>1.
\]

\section{Applications}\label{sec:minty}
We first deduce Corollary~\ref{cor:minty} from
Theorems~\ref{thm:main} and~\ref{thm:sharp}. The essential
point is that the Jacobian of the Minty change of variables
is comparable to the Hessian trace.

\begin{lemma}\label{lem:minty-identities}
Under the assumptions in Corollary~\ref{cor:minty}, \(S,V,g\) are defined by
\eqref{eq:minty-definition}. Then \(S:\Omega\to V\) is a
homeomorphism, and both \(S^{-1}\) and \(g\) are \(1\)-Lipschitz.
At almost every \(x\in\Omega\), set
\[
 H=D^2u(x),\qquad \sigma=\tr H,\qquad
 d=\det H,\qquad J=1+\sigma+d.
\]
Then, almost everywhere in \(V\),
\begin{equation}\label{eq:minty-wirtinger}
 |g_z(S(x))|=\frac{\sqrt{\sigma^2-4d}}{J},\qquad
 g_{\bar z}(S(x))=\frac{d-1}{J}.
\end{equation}
For every nonnegative Borel function \(F\),
\begin{equation}\label{eq:minty-area}
 \int_V F(z)\,dz
 =\int_\Omega F(S(x))J(x)\,dx.
\end{equation}
Moreover, there exist constants \(c_\lambda,C_\lambda>0\)
such that
\begin{equation}\label{eq:minty-comparable}
 \sigma\le J\le C_\lambda \sigma,\qquad
 \frac{c_\lambda}{\sigma}\le\delta(S(x))
 \le\frac{C_\lambda}{\sigma}
 \quad\text{a.e. in }\Omega.
\end{equation}
\end{lemma}

\begin{proof}
The planar regularity theory and Theorem~\ref{thm:main} imply that
\[
 u\in C^1_{\mathrm{loc}}(\Omega)
 \cap W^{2,1}_{\mathrm{loc}}(\Omega).
\]
The Alexandrov equation implies \(\lambda\le d\le\lambda^{-1}\)
almost everywhere. Convexity implies
\[
 (x-y)\cdot(\nabla u(x)-\nabla u(y))\geq0.
\]
Consequently,
\[
 |S(x)-S(y)|^2\geq|x-y|^2,\qquad
 |g(S(x))-g(S(y))|^2\leq|S(x)-S(y)|^2.
\]
It follows that \(S\) is injective, \(S^{-1}\) is
\(1\)-Lipschitz on \(V\), and \(g\) is \(1\)-Lipschitz.
Brouwer's invariance of domain theorem shows that \(V\) is open and \(S\) is
a homeomorphism onto \(V\).

Let \(q(x)=|x|^2/2\). For smooth convex \(u_j\), the
two-dimensional determinant identity reads
\[
 \det D^2(u_j+q)=1+\Delta u_j+\det D^2u_j.
\]
Approximate \(u\) locally by smooth convex functions.
Weak continuity of Alexandrov measures and distributional
convergence of the Laplacians imply the measure identity
\[
 \mu_{u+q}=(1+\sigma+d)\,dx=J\,dx.
\]
Since \(u+q\) is a \(C^1\) convex function and
\(\nabla(u+q)=S\), its Alexandrov measure satisfies
\[
 \mu_{u+q}(E)=|S(E)|
\]
for every Borel set \(E\Subset\Omega\), and hence for
arbitrary Borel sets by exhaustion. This proves
\eqref{eq:minty-area} for indicators and then for nonnegative
Borel functions.

At a differentiability point of \(\nabla u\), the chain rule
and the locally Lipschitz inverse of \(S\) yield
\[
 Dg(S(x))=
 \begin{pmatrix}1&0\\0&-1\end{pmatrix}\mathcal R,
 \qquad
 \mathcal R=(H-I)(H+I)^{-1}.
\]
The matrix \(\mathcal R\) is symmetric. Thus the Wirtinger derivatives
are
\[
 g_z=\frac{\mathcal R_{11}-\mathcal R_{22}}{2}-i\mathcal R_{12},\qquad
 g_{\bar z}=\frac{\tr\mathcal R}{2}.
\]
If \(\kappa_1\ge\kappa_2>0\) are the eigenvalues of \(H\),
the eigenvalues of \(\mathcal R\) are
\((\kappa_i-1)/(\kappa_i+1)\).
Since \(\det(H+I)=J\), their difference and sum yield
\eqref{eq:minty-wirtinger}. Since \(S\) maps null sets to null sets, these identities also hold
almost everywhere in the \(z\)-coordinates.

Finally \(\sigma\ge2\sqrt d\ge2\sqrt\lambda\), so
\(\sigma\le J\le C_\lambda \sigma\), and
\[
 \delta(S(x))
 =\frac{1+d+\sigma-\sqrt{\sigma^2-4d}}{J}.
\]
Because
\(0\le \sigma-\sqrt{\sigma^2-4d}\le2\sqrt d\) and
\(\lambda\le d\le\lambda^{-1}\), the numerator lies
between \(1+\lambda\) and
\((1+\lambda^{-1/2})^2\). This proves
\eqref{eq:minty-comparable}, including \(\delta>0\) almost
everywhere.
\end{proof}

\begin{proof}[Proof of Corollary~\ref{cor:minty}]
Fix an open set \(V'\Subset V\) and set
\(K=S^{-1}(\overline{V'})\Subset\Omega\).
By Theorem~\ref{thm:main},
\[
 |\{x\in K:\sigma(x)>R\}|\leq C_KR^{-P}
 \qquad(R>0).
\]
Since \(P>1\), the layer-cake formula yields
\begin{equation}\label{eq:minty-layercake}
\begin{aligned}
 \int_{K\cap\{\sigma>R\}}\sigma\,dx
 &=R|\{x\in K:\sigma(x)>R\}|\\
   &\quad+\int_R^\infty|\{x\in K:\sigma(x)>t\}|\,dt\\
 &\le \frac{P}{P-1}C_KR^{1-P}.
\end{aligned}
\end{equation}
If \(\delta(S(x))<t\), then
\eqref{eq:minty-comparable} implies \(\sigma(x)>c_\lambda/t\).
Using \eqref{eq:minty-area}, \eqref{eq:minty-comparable},
and \eqref{eq:minty-layercake}, we obtain
\[
 |\{z\in V':\delta(z)<t\}|
 \le C_\lambda\int_{K\cap\{\sigma>c_\lambda/t\}}\sigma\,dx
 \le C\,t^{P-1}.
\]
This proves the asserted small-value estimate. Substituting \(t=R^{-1}\) and using \(|V'|<\infty\) for bounded
\(R\) yields \(\delta^{-1}\in L^{P-1,\infty}(V')\).

For sharpness, let \(u\) be the example in
Theorem~\ref{thm:sharp}, with \(Q=(0,1)^2\).
Here \(S:\R^2\to\R^2\) is a global homeomorphism:
it agrees with \(2x\) outside \(Q\), so it is proper.
Choose an open \(V'\Subset\R^2\) containing
\(S(\overline Q)\). From the upper bound in
\eqref{eq:minty-comparable},
\(\sigma>C_\lambda/t\) implies \(\delta(S(x))<t\).
Consequently, for all small \(t>0\),
\begin{align*}
 |\{z\in V':\delta(z)<t\}|
 &\ge\int_{Q\cap\{\sigma>C_\lambda/t\}}J\,dx\\
 &\ge\frac{C_\lambda}{t}
        |\{x\in Q:\sigma(x)>C_\lambda/t\}|\\
 &\ge c\,t^{P-1},
\end{align*}
where the last inequality follows from
\eqref{eq:two-sided}. This proves \eqref{eq:minty-sharp}.
Moreover,
\(\bigl|\{z\in V':\delta^{-1}(z)>R\}\bigr|
 \asymp R^{-(P-1)}\) for all sufficiently large \(R\).
For any finite \(b>0\), the tail contribution to the
\(b\)-th power of the \(L^{P-1,b}\) quasi-norm contains
\[
 \int_{R_0}^\infty
 \left[R^{P-1}|\{\delta^{-1}>R\}\cap V'|\right]^{b/(P-1)}
 \frac{dR}{R}=\infty.
\]
This proves the Lorentz sharpness.
\end{proof}

We then prove Corollary~\ref{cor:bilip} by applying
Theorem~\ref{thm:main} to \(u\) and its Legendre transform \(u^*\), followed by a two-sided maximal-function restriction.

\begin{lemma}\label{lem:gradient-inverse}
Under the assumptions in Corollary~\ref{cor:bilip}, set
\(T=\nabla u\), \(W=T(\Omega)\). Then
\(T:\Omega\to W\) is a homeomorphism, and its inverse is
\(T^{-1}=\nabla u^*\), where
\[
 u^*(y)=\sup_{x\in\Omega}\{x\cdot y-u(x)\},
 \qquad y\in W.
\]
In the local Alexandrov sense on \(W\),
\begin{equation}\label{eq:dual-ma}
 \mu_{u^*}=\frac{dy}{d\circ T^{-1}},
 \qquad \lambda\le d\circ T^{-1}\le\lambda^{-1},
 \qquad d=\det D^2u\quad\text{a.e.}
\end{equation}
Consequently,
\begin{equation}\label{eq:inverse-weak}
 DT\in L^{P,\infty}_{\mathrm{loc}}(\Omega),\qquad
 D(T^{-1})\in L^{P,\infty}_{\mathrm{loc}}(W),
\end{equation}
and, for all Borel \(E\subset\Omega\), \(F\subset W\),
\begin{equation}\label{eq:gradient-measure}
 \lambda|E|\le|T(E)|\le\lambda^{-1}|E|,
 \qquad \lambda|F| \le |T^{-1}(F)|\le\lambda^{-1}|F|.
\end{equation}
\end{lemma}

\begin{proof}
The map \(T\) is continuous and injective. Brouwer's invariance of
domain theorem therefore shows that \(T\) is a homeomorphism onto \(W\). Convex duality shows that
\(u^*\) is finite and \(C^1\) on \(W\), with
\(\nabla u^*=T^{-1}\) there.
The Alexandrov measure identity and injectivity yield
\[
 |T(E)|=\mu_u(E)=\int_E d(x)\,dx.
\]
In particular, for every nonnegative Borel \(\varphi\),
\[
 \int_W\varphi(y)\,dy
 =\int_\Omega\varphi(T(x))d(x)\,dx.
\]
Applying this identity to
\(\varphi=\mathbf 1_F/(d\circ T^{-1})\) yields
\[
 |T^{-1}(F)|=\int_F\frac{dy}{d(T^{-1}(y))}.
\]
Since \(T^{-1}=\nabla u^*\), this is precisely the
Alexandrov identity \eqref{eq:dual-ma}. It follows that
\(\lambda\,dy\le\mu_{u^*}\le\lambda^{-1}\,dy\)
locally on \(W\). Applying Theorem~\ref{thm:main} to
\(u^*\) on balls compactly contained in \(W\) and to
\(u\) on \(\Omega\) proves \eqref{eq:inverse-weak}.
The two measure identities and the density bounds imply
\eqref{eq:gradient-measure}.
\end{proof}

\begin{lemma}
\label{lem:weak-lusin}
Let \(1<p<\infty\) and let
\(F:\R^2\to\R^2\) be continuous, compactly supported,
and in \(W^{1,1}(\R^2)\), with
\(|DF|\in L^{p,\infty}(\R^2)\). There is a full-measure
set \(\mathcal L_F\) such that, for all
\(x,y\in\mathcal L_F\),
\begin{equation}\label{eq:pointwise-sobolev}
 |F(x)-F(y)|\le C|x-y|
 \bigl(M|DF|(x)+M|DF|(y)\bigr),
\end{equation}
where \(M\) is the Hardy--Littlewood maximal operator.
For every \(L>0\), the restriction of \(F\) to
\[
 G_L(F)=\bigl\{x\in\mathcal L_F:
          M|DF|(x)\le L/(2C)\bigr\}
\]
is \(L\)-Lipschitz and
\begin{equation}\label{eq:weak-lusin-tail}
 |\R^2\setminus G_L(F)|
 \le C_p\|DF\|_{L^{p,\infty}}^pL^{-p}.
\end{equation}
\end{lemma}

\begin{proof}
At Lebesgue points, telescope the averages of \(F\) over
\(B_{2^{-j}r}(x)\) and apply the \(L^1\) Poincar\'e
inequality to obtain
\[
 |F(x)-F_{B_r(x)}|\le Cr\,M|DF|(x).
\]
For \(r=|x-y|\), comparing the averages over
\(B_r(x)\) and \(B_r(y)\) through a containing ball
proves \eqref{eq:pointwise-sobolev}; see also
\cite{BojarskiHajlasz1993}. Continuity identifies the
pointwise values with the chosen representative.

Split \(|DF|=|DF|\mathbf1_{\{|DF|\le t/2\}}+
 |DF|\mathbf1_{\{|DF|>t/2\}}\). Since the maximal function
of the first summand does not exceed \(t/2\), the weak
\((1,1)\) maximal inequality and the distribution bound for
\(|DF|\) imply
\[
 |\{M|DF|>t\}|
 \le\frac{C}{t}\int_{\{|DF|>t/2\}}|DF|
 \le C_p\|DF\|_{L^{p,\infty}}^pt^{-p}.
\]
Here \(p>1\) ensures convergence of the tail integral.
Choosing \(t=L/(2C)\) proves
\eqref{eq:weak-lusin-tail}, and
\eqref{eq:pointwise-sobolev} yields the Lipschitz restriction.
\end{proof}

\begin{proof}[Proof of Corollary~\ref{cor:bilip}]
Fix \(K\Subset\Omega\), and choose
\(\chi\in C_c^\infty(\Omega)\) and
\(\eta\in C_c^\infty(W)\) equal to \(1\) on
neighborhoods of \(K\) and \(T(K)\), respectively.
Extend
\[
 F=\chi T,\qquad G=\eta T^{-1}
\]
by zero to \(\R^2\). Both are continuous,
compactly supported Sobolev maps; by
Lemma~\ref{lem:gradient-inverse} and the product rule,
\(DF,DG\in L^{P,\infty}(\R^2)\).
Applying Lemma~\ref{lem:weak-lusin} to each of them
(with \(p=P\)), we obtain
\[
 |K\setminus G_L(F)|+|T(K)\setminus G_L(G)|\le CL^{-P}.
\]
Define the common good set
\[
 E_L^0=K\cap G_L(F)\cap T^{-1}\bigl(G_L(G)\bigr).
\]
By \eqref{eq:gradient-measure},
\[
 |K\setminus E_L^0|
 \le |K\setminus G_L(F)|
 +\lambda^{-1}|T(K)\setminus G_L(G)|
 \le CL^{-P}.
\]
For \(x,y\in E_L^0\), \(T=F\) is \(L\)-Lipschitz,
while \(T^{-1}=G\) is \(L\)-Lipschitz on
\(T(E_L^0)\). Hence
\[
 L^{-1}|x-y|\le|T(x)-T(y)|\le L|x-y|.
\]
By inner regularity choose a compact
\(E_L\subset E_L^0\) with
\(|E_L^0\setminus E_L|\le L^{-P}\).
The distance inequalities persist, and injectivity and
\eqref{eq:gradient-measure} imply
\[
 |T(K)\setminus T(E_L)|
 =|T(K\setminus E_L)|
 \le\lambda^{-1}|K\setminus E_L|.
\]
This proves \eqref{eq:bilip-exception} and
\eqref{eq:bilip-on-good}.

For sharpness, take the example in
Theorem~\ref{thm:sharp}. Choose a compact \(K\Subset Q\) whose
interior contains the compact null set outside which \(u\) is smooth.
The nested marked sets in the construction eventually lie
inside \(K\); hence \(u\) coincides with one fixed smooth
stage outside \(K\), and \(D^2u\) is bounded on \(Q\setminus K\).
Therefore the lower bound in \eqref{eq:two-sided} holds
with \(K\) in place of \(Q\), at all sufficiently large levels.

Suppose \(E\subset K\) is measurable and
\(\operatorname{Lip}(T|_E)\le L\). Kirszbraun's extension
theorem extends \(T|_E\) to an \(L\)-Lipschitz map on
\(\R^2\). Locality of weak derivatives on the set where
this extension agrees with \(T\) yields
\(\|DT\|_{\mathrm{op}}\le L\) almost everywhere on \(E\).
Since \(DT=D^2u\) is positive semidefinite,
\(\Delta u\le2L\) almost everywhere on \(E\). Thus
\[
 |K\setminus E|
 \ge|\{x\in K:\Delta u(x)>2L\}|
 \ge cL^{-P},
\]
which is \eqref{eq:bilip-sharp}.
The same obstruction holds in the target area, since
\(|T(K\setminus E)|\ge\lambda|K\setminus E|\).
\end{proof}

Corollary~\ref{cor:bilip} is closely related to a question of interest
in optimal transport: how regular a transport map is on a set carrying most of the mass.
Suppose \(\Omega_0,\Omega_1\subset\R^2\) are bounded open
convex domains, and let \(\mu_i=\rho_i\,dx\) be probability
measures with \(0<a\le\rho_i\le b\), \(a<b\).
The Brenier map \(T=\nabla u\) from \(\mu_0\) to
\(\mu_1\) has an interior homeomorphic representative by
the planar regularity theory for convex transport domains.
Its Alexandrov density satisfies
\[
 \det D^2u=\frac{\rho_0}{\rho_1\circ T},\qquad
 \frac ab\le\det D^2u\le\frac ba.
\]
Apply Corollary~\ref{cor:bilip} with \(\lambda=a/b\).
Since \(\rho_0\le b\), for every compact
\(K\Subset\Omega_0\) and all sufficiently large \(L\)
there is a compact \(E_L\subset K\) for which
\[
 \mu_0(K\setminus E_L)\le CL^{-(a+b)/(b-a)}.
\]
The pushforward relation and injectivity yield
\[
 \mu_1(T(K)\setminus T(E_L))
 =\mu_0(K\setminus E_L),
\]
proving \eqref{eq:transport-mass}. In particular, for each fixed
compact \(K\Subset\Omega_0\) and all sufficiently small
\(\varepsilon>0\), one can remove at most \(\varepsilon\) of the
source mass in \(K\), together with the corresponding target mass
in \(T(K)\), so that the restriction of \(T\) to the retained set
is bi-Lipschitz with constant at most
\(C_K\varepsilon^{-(b-a)/(a+b)}\).

To see sharpness within the transport class, use the example
of Theorem~\ref{thm:sharp}. It agrees with \(|x|^2/2\)
outside \(Q\); hence \(T=\nabla u\) is a proper
homeomorphism of \(\R^2\), agrees with the identity on
\(\R^2\setminus Q\), and satisfies \(T(Q)=Q\).
Write \(d=\det D^2u\) almost everywhere on \(Q\), and set
\[
 Z=\int_Q\sqrt{d(x)}\,dx,\qquad
 \rho_0(x)=\frac{\sqrt{d(x)}}{Z},\qquad
 \rho_1(y)=\frac{1}{Z\sqrt{d(T^{-1}(y))}}.
\]
By the change-of-variables identity, both densities have unit integral
and satisfy
\[
 \frac{\sqrt\lambda}{Z}\le\rho_i\le
 \frac{1}{Z\sqrt\lambda},\qquad i=0,1.
\]
Moreover,
\[
 T_\#(\rho_0\,dx)=\rho_1\,dy.
\]
Thus the common ratio of lower to upper density bounds
is exactly \(\lambda\), and \(T\) is the quadratic-cost
Brenier map for these measures. Choose the compact set
\(K\Subset Q\) in the sharpness proof above. For every
measurable \(E\subset K\) with
\(\operatorname{Lip}(T|_E)\le L\),
\[
 \mu_0(K\setminus E)
 \ge\frac{\sqrt\lambda}{Z}|K\setminus E|
 \ge cL^{-P},\qquad
 \mu_1(T(K)\setminus T(E))=\mu_0(K\setminus E).
\]
Hence the transport exceptional-mass exponent is optimal.

\section*{Use of artificial intelligence}
During the development and preparation of this work, the authors
used GPT for mathematical exploration. Through exchanges with the
authors, GPT generated drafts of the proofs and the counterexample
construction. These outputs were treated as proposals requiring
verification, rather than as authoritative evidence of correctness.
The authors checked the mathematical arguments, revised and
reorganized the proofs, and supplied missing details. The authors
take full responsibility for the correctness, originality, and
integrity of the final manuscript, including its mathematical
statements, proofs, computations, references, and the appropriate
attribution of prior work.



\begingroup
\small
\raggedright
\bibliographystyle{amsplain}
\bibliography{references}
\endgroup
\printauthorinformation

\end{document}